\documentclass[12pt,a4paper,reqno]{article}

\usepackage[latin,english]{babel}
\usepackage[utf8]{inputenc}
\usepackage[T1]{fontenc}
\usepackage{eufrak, color, mathrsfs,dsfont}
\usepackage{times}
\usepackage{geometry}
\usepackage{amsmath}
\usepackage{amssymb}
\usepackage{amsthm}
\usepackage{tikz}
\usepackage{mathtools}
\usepackage{mathabx}
\usepackage{cases}
\usepackage{braket}
\usepackage[toc,page]{appendix}
\usepackage{pdfsync}
\usepackage{hyperref}
\usepackage{psfrag}
\mathtoolsset{showonlyrefs}
\usepackage[mathscr]{euscript}
\usepackage{math}
\numberwithin{equation}{section}

\theoremstyle{plain}
\newtheorem{thm}{Theorem}[section]

\newtheorem{lem}[thm]{Lemma}
\newtheorem{prop}[thm]{Proposition}

\theoremstyle{definition}
\newtheorem{defn}[thm]{Definition}

\theoremstyle{remark}
\newtheorem{rem}[thm]{Remark}

\DeclareMathOperator{\spt}{spt}

\newcommand{\pa}{\partial}

\title{Time almost-periodic solutions of the incompressible Euler equations}

\author{Luca Franzoi, Riccardo Montalto}

\date{}

\begin{document}

\maketitle

%\date{\today}

\begin{abstract}
We construct time almost-periodic solutions (global in time) with finite regularity to the incompressible Euler equations on the torus $\T^d$, with $d=3$ and $d\in\N$ even.
\end{abstract}



%\tableofcontents

\section{Introduction}

The goal of this paper is to construct \emph{time almost-periodic} solutions (infinite dimensional invariant tori) of the Euler equations
\begin{equation}\label{euler.eq}
\begin{aligned}
& 	\pa_{t} u + u\cdot \nabla u + \nabla p = 0 \,, \quad {\rm div} \,u = 0\,, \\
& u : \R \times \T^d \to \R^d\,, \quad p : \R \times \T^d \to \R\,,
	\end{aligned}
\end{equation}
on the $d$-dimensional torus $\T^d$, $\T:= \R / 2\pi\Z$, where either $d=3$ or $d\geq 2$ is any even positive integer. 
These solutions extend the works ot Crouseille \& Faou \cite{CF} (in dimension 2) and Enciso, Peralta Salas \& Torres de Lizaur \cite{EPsTdL} (in dimensions 3 or even) from time quasi-periodic to time almost-periodic solutions.
In fact, the construction here follows closely the one in \cite{EPsTdL}.

We need to specify how a smooth solution of the Euler equations \eqref{euler.eq} is called \emph{almost-periodic} in this paper. We need some preliminaries.

Let $\cC_{\rm div}^{s}(\T^d,\R^d)$, with $s\in\N \cup\{+\infty\}$, be the space of $\cC^s$-smooth, divergence free $d$-dimensional vector fields on $\T^d$. This space is a Banach space if $s< \infty$ and a Fr\'echet space when $s=\infty$. We endow it with the system of seminorms $(\|\,\cdot\,\|_{n,\infty})_{n\in\{0,1,...,s\}}$ defined by 
\begin{equation}
	 \|f\|_{n,\infty} := \sup_{x\in \T^d} \max_{\begin{subarray}{c}
	 \alpha\in \N_0^d \\
	  |\alpha|=n
	  \end{subarray}} |\pa_{x}^\alpha f(x) | \,, \quad n=0,1,...,s\,;
\end{equation}
throughout the paper, for sake of simplicity in the notation, $|\,\cdot \,|$ denotes the standard Euclidean norm, without specifying the dimension of the evaluated object, which will be clear from the context each time. We denote by $\ell^\infty(\N,\N)$ the set of sequences in $\N$ that are bounded.
 Let $(J_k)_{k\in\N}\in \ell^\infty(\N,\N)\setminus\{0\}$ be given and, for a fixed $m\in\{1,...,d-1\}$, we define the sequence
\begin{equation}\label{sequence.Nk}
	(N_k)_{k\in\N} \in \ell^\infty(\N,\N)\,, \quad N_k:=(d-m) J_{k} \in \N \,, \quad k\in\N\,.
\end{equation}
We define the infinite dimensional torus $(\T^{N_k})_{k\in\N}$ and its "tangent space" $(\R^{N_k})_{k\in\N}$ as 
\begin{equation}
\begin{aligned}
		& (\T^{N_k})_{k\in\N} := \big\{ \theta = (\theta_{k})_{k\in\N} \, :\, \theta_{k} \in \T^{N_k} , \ | \theta|_{\infty} < \infty  \big\}\,, \\
		& 	(\R^{N_k})_{k\in\N} := \big\{ \nu = (\nu_{k})_{k\in\N} \, :\, \nu_{k} \in \R^{N_k}  , \ | \nu |_{\infty} < \infty   \big\}\,,
\end{aligned}
\end{equation}
where we defined $| \nu |_{\infty}:= \sup_{k\in\N} |\nu_{k}|$. Note that, since $(N_{k})_{k\in\N}$ is bounded, then $| \theta |_{\infty}\leq (2\pi)^{\| (N_k)\|_{\ell^\infty}} < \infty$ for any sequence $\theta = (\theta_{k})_{k\in\N}$.

\begin{defn}
	Let $s\in\N \cup \{+\infty\}$. We say that $u(t,x)$ is \emph{time almost-periodic}
	if there exists a sequence  of vectors $\nu \in (\R^{N_k})_{k\in\N} $ and a $\cC^1$-smooth embedding $U:(\T^{N_k})_{k\in\N}\to \cC_{\rm div}^{s} (\T^d,\R^d)$ such that the velocity field $u(t,x)$ can be written as
	\begin{equation}\label{u(t).as.embedding}
		u(t,\,\cdot\,) = U(\vartheta)|_{\vartheta=\theta+\nu t} \,, \quad \text{ for some} \quad \theta\in (\T^{N_k})_{k\in\N} \,,
	\end{equation}
	and the  sequence of frequency vectors $\nu=(\nu_{k})_{k\in\N}$ is \emph{non-resonant}, meaning that
%	% for any $N\in\N$ and any $K\subset \N$, $|K|=N$, we have that $\sum_{k\in K} \nu_{k} \ell_{k}\neq 0 $ for any $(\ell_{k})_{k\in K} \in \Z^N$.
%	\begin{equation}\label{nonres}
%		\sum_{k\in \N} \nu_{k} \cdot \ell_{k}\neq 0 \quad \ \forall\, \ell_{k} \in \Z^{N_k}, \ k \in \N, \ \text{ with } \ \ (\ell_{k})_{k\in \N} \neq (0)_{k\in\N}  \,.
%	\end{equation}
\begin{equation}\label{nonres}
	\sum_{k\in \N} \nu_{k} \cdot \ell_{k} \neq 0\,, \quad \forall\,\ell=(\ell_{k})_{k\in\N} \in (\Z^{N_k})_{k\in\N} \ \ \text{ with } \ \ 0< |\ell|_{\eta} <\infty  \,,
\end{equation}
where, for a fixed $\eta>0$, we define $|\ell|_{\eta} := \sum_{k\in \N} k^\eta |\ell_{k}| $.
%\begin{equation}
%	|\ell|_{\eta} := \sum_{k\in \N} k^\eta |\ell_{k}| \,.
%\end{equation}
Note that $|\ell|_{\eta}< \infty$ implies that $\ell_{k}\neq 0 \in \Z^{N_k}$ only for finitely many $k\in\N$.
\end{defn}

\begin{defn}\label{C1-smooth.def}
	By saying that the map  $U:(\T^{N_k})_{k\in\N}\to \cC_{\rm div}^{s} (\T^d,\R^d)$ is $\cC^1_b$ (${\mathcal C}^1$ and bounded), we mean that, $U$ is a Frechet-differentiable map with continuous Frechet derivative and for any $n\in \{ 0, 1, \ldots, s\}$, there exists a constant $C_n>0$ such that
%	\begin{equation}
%	\begin{aligned}
%		& \sup_{\vartheta \in (\T^{N_k})_{k\in\N}} \|  U(\vartheta) \|_{n,\infty} \leq C_n \,, \\
%		&	\sup_{\vartheta \in (\T^{N_k})_{k\in\N}} \| \di_{\vartheta} U(\vartheta) [\nu] \|_{n,\infty} \leq C_n | \nu |_{\infty} \quad \forall\,\nu \in (\R^{N_k})_{k\in\N}\,, 
%	\end{aligned}
%	\end{equation}
	\begin{equation}
	\begin{aligned}
		& \sup_{\vartheta \in (\T^{N_k})_{k\in\N}} \|  U(\vartheta) \|_{n,\infty} \leq C_n   \,,  \\
		&	\sup_{\vartheta \in (\T^{N_k})_{k\in\N}} \| \di_{\vartheta} U(\vartheta) [\wh\vartheta] \|_{n,\infty} \leq C_n | \wh\vartheta |_{\infty} \quad \forall\, \wh\vartheta \in (\R^{N_k})_{k\in\N}\,, 
	\end{aligned}
\end{equation}
	where the linear operator $\di_{\vartheta} U(\vartheta) : ( \R^{N_k})_{k\in\N} \to \cC_{\rm div}^{s} (\T^d,\R^d)$ is the Fr\'echet differential of the embedding $U(\vartheta)$ at the point $\vartheta\in (\T^{N_k})_{k\in\N}$. 
\end{defn}
With this definition of a $\cC^1_{b}$ embedding, we have that the function $u(t,\,\cdot\,)$ is $\cC^1$ with respect to $t\in\R$, by \eqref{u(t).as.embedding} and $\pa_{t} u(t,\,\cdot\,) = \di_{\vartheta} U(\theta + \nu t) [\nu] \in \cC_{\rm div}^{s}(\T^d,\R^d)$.

We will look for solutions where the embedding $U$ is \emph{non-symmetric}, or \emph{non-traveling}, in the sense that, for any $\vartheta\in(\T^{N_k})_{k\in\N}$, the divergence-free vector field $U(\vartheta)$ is not invariant under any 1-parameter group of translations on $\T^d$. In this way, we ensure that the solution $u(t,x)$ depends effectively on all $d$ coordinates and we do not have any reduction to solutions of lower dimensions by traveling directions.

The statement of the main result is the following.

\begin{thm}\label{teorema}
	{\bf (Time almost-periodic solutions of the Euler equations).}
	Assume that the dimension $d$ is either 3 or even. Let $S\in \N$ be fixed. There exists $\varepsilon_{0}\in (0,1)$ small enough such that, for any $\varepsilon\in (0,\varepsilon_{0})$ and for any sequence of frequencies $\nu \in (\R^{N_k})_{k\in\N}\setminus\{0\}$ satisfying
	\begin{equation}\label{small.freq}
		\sup_{k\in \N} \varepsilon^{-(S + 1)(k-1)} |\nu_{k}| <\infty \,,
	\end{equation}
	there exists a non-symmetric $\cC^{1}_{b}$ embedding $U:(\T^{N_k})_{k\in\N}\to \cC_{\rm div}^S(\T^d,\R^d)$ and a family of initial data $u_{\theta}\in \cC_{\rm div}^S(\T^d,\R^d)$, $\theta\in (\T^{N_k})_{k\in\N}$, such that $u(t,\,\cdot\,)= U(\theta+\nu t)$, with $u(0,\,\cdot\,)=u_{\theta}$, is a solution of \eqref{euler.eq} with pressure $p(t,\,\cdot\,)= P(\theta+\nu t)$, where
	\begin{equation}
		P(\vartheta) :=(-\Delta)^{-1}\big[ {\rm div} (U(\vartheta)\cdot \nabla U(\vartheta)  )\big] : (\T^{N_k})_{k\in\N} \to \cC^S(\T^d)\,.
 	\end{equation}
	As a consequence, if the sequence $\nu=(\nu_{k})_{k\in\N}$ is non-resonant, namely it satisfies \eqref{nonres}, then the solution $u(t,x)$ is time almost-periodic.  
\end{thm}

\begin{rem}
	As it will be clear from the construction in the following section, the embedding $U(\vartheta)$ is determined as a combination of infinitely many embedding $U_k(\vartheta_{k})$, with $\vartheta_{k}\in \T^{N_k}$, which coincides with the embedding constructed in \cite{EPsTdL}, with the size of the embedding $U_{k}$ becoming smaller and smaller  as $k\to \infty$. The major difference in the analysis with respect to \cite{EPsTdL} is that we have to effectively prove the smoothness of the embedding and the regularity of the vector field. This is not trivial. 
\end{rem}

\begin{rem}
	The condition \eqref{nonres} of irrationality for the sequence of frequencies $\nu\in(\R^{N_k})_{k\in\N}\setminus \{0\}$ is not necessary in the construction of the embedding $U$. Depending on relations between all the frequencies, we may obtain embedding for lower dimensional tori, either finite dimensional (quasi-periodic or periodic) or still infinite dimensional (that is, almost-periodic). On the other hand, the control on the frequency vectors in \eqref{small.freq} is required to ensure that the solution $u(t,x)$ is indeed a finitely smooth vector field and a simpler control on the norm $|\nu|_\infty$ is not enough. At the physical level, is also implies that we obtain solutions whose leading order frequencies of oscillations are only finitely many and the almost-periodicity in time is due to the presence of infinitely oscillations with smaller and smaller frequencies.
%	The condition can be even restricted to obtain an infinitely smooth solution.
\end{rem}

\paragraph{Related results.} 
In the last years, there has been a discrete surge of works proving the existence of time quasi-periodic waves for PDEs arising in fluid dynamics.
With the exception of the aforementioned works \cite{CF} and \cite{EPsTdL}, there type of results in literature are proved by means of KAM for PDEs techniques, to deal with the presence of small divisors issues and consequent losses of regularity.
For the two dimensional water waves equations, we mention
%Iooss, Plotnikov & Toland [27] for periodic standing waves, 
Berti \& Montalto \cite{BM}, Baldi, Berti, Haus \& Montalto \cite{BBHM} for time quasi-periodic standing waves and Berti, Franzoi \& Maspero \cite{BFM},
\cite{BFM21}, Feola \& Giuliani \cite{FeoGiu} for time quasi-periodic traveling wave solutions. 
Recently, the existence of time quasi-periodic solutions was proved for the contour dynamics of vortex patches in active scalar equations. 
We mention Berti, Hassainia \& Masmoudi \cite{BertiHassMasm} for vortex patches of
the Euler equations close to Kirchhoff ellipses, Hmidi \& Roulley \cite{HmRo} for the 
quasi-geostrophic shallow water equations,  Hassainia, Hmidi \& Masmoudi \cite{HaHmMa} for generalized surface quasi-geostrophic equations, Roulley \cite{Roulley} for Euler-$\alpha$ flows, Hassainia \& Roulley \cite{HaRo} for Euler equations in the unit disk close to Rankine vortices and Hassainia, Hmidi \& Roulley \cite{HaHmRou} for 2D Euler annular vortex patches.
%We mention Berti, Hassainia \& Masmoudi \cite{BertiHassMasm} for vortex patches of
%the Euler equations close to Kirchhoff ellipses, Hmidi \& Roulley \cite{HmRo} for the 
%quasi-geostrophic shallow water equations,  Hassainia, Hmidi \& Masmoudi \cite{HaHmMa} for generalized surface quasi-geostrophic equations, Roulley \cite{Roulley} for Euler-$\alpha$ flows, Hassainia \& Roulley \cite{HaRo} for Euler equations in the unit disk close to Rankine vortices and Hassainia, Hmidi \& Roulley \cite{HaHmRou} for 2D Euler annular vortex patches.
%All the aforementioned results concern 2D Euler equations.
Time quasi-periodic solutions were also constructed for the 3D Euler equations with time quasi-periodic external force \cite{BM20} and for the forced 2D Navier-Stokes
equations \cite{FrMo} approaching in the zero viscosity limit time quasi-periodic solutions of the 2D Euler equations for all times. 

The existence of other non-trivial invariant structures is also a topic of interest in  fluid dynamics. In particular, for the Euler equations in two dimension close to shear flows, we mention the works by Lin \& Zeng \cite{LZ} and Castro \& Lear \cite{CL21} for periodic traveling waves close the Couette flow,  by Coti Zelati, Elgindi \& Widmayer \cite{CzEW} for stationary waves around non-monotone shears, by  Franzoi, Masmoudi \& Montalto \cite{FrMaMo} for quasi-periodic traveling waves close to the Couette flow, and the recent work by Castro \& Lear \cite{CL23} for time periodic rotating solutions close to the Taylor-Couette flow.

Concerning the existence of almost periodic solutions by means of KAM methods, we mention P\"oschel \cite{PoeschelAP}, Bourgain \cite{Bou04}, Biasco-Massetti-Procesi \cite{MaProBia}, \cite{BMP2} and Corsi-Gentile-Procesi \cite{CGP}. In all these results the authors consider semilinear NLS type equations with external parameters. For PDEs with unbounded perturbations (with external parameters as well) we mention Montalto-Procesi \cite{MoPr021} and Corsi-Montalto-Procesi \cite{CoMoPro}.

\noindent
We remark that our result is the first one concerning existence of almost-periodic solutions for an autonomous quasi-linear PDEs in higher space dimension and it is obtained with non-KAM techniques. 
%We finally mention that time quasi-periodic solutions for the Euler equations were constructed also by Crouseilles \& Faou \cite{CF} in 2D, with a very recent extension by Enciso, Peralta-Salas \& de Lizaur \cite{EPsTdL} in 3D and even dimensions: we remark that these latter solutions are engineered so that there are no small divisors issues to deal with, with consequently much easier proofs and a drawback of not having information on the eventual stability of the solutions. 

\bigskip

\medskip

%	\paragraph{Acknowledgments.} 
\noindent
{\bf Acknowledgments. }
	The work of the authors Luca Franzoi and Riccardo Montalto is funded by the European Union, ERC STARTING GRANT 2021, "Hamiltonian Dynamics, Normal Forms and Water Waves" (HamDyWWa), Project Number: 101039762. Views and opinions expressed are however those of the authors only and do not necessarily reflect those of the European Union or the European Research Council. Neither the European Union nor the granting authority can be held responsible for them. 

\noindent
The work of the author Riccardo Montalto is also supported by PRIN 2022 "Turbulent effects vs Stability in Equations from Oceanography", project number: 2022HSSYPN. 

\noindent
Riccardo Montalto is also supported by INDAM-GNFM. 

\bigskip
\medskip

%\paragraph{Notations.} 
\noindent
{\bf Notations.}
In this paper, we use the following notations:
%\\[1mm]
%\noindent $\bullet$ $\N := \{1,2,...,n,...\}$ and $\N_0:= \{0\}\cup \N$;
%\\[1mm]
%\noindent $\bullet$ $|\,\cdot\,|$ denotes the Euclidean norm in the dimension of the evaluated object, without any further specification;
\\[1mm]
\noindent $\bullet$ $B_{d,\rho}(\tp) := \{ x\in\R^d \,:\, |x-\tp| < \rho \}$, with $\tp \in \R^d$ and $\rho >0$;
%denoted the open ball in $\R^d$ (or eventually $\T^d$) centered at $\tp \in \R^d$ with radius $\rho>0$.
%\\[1mm]
%\noindent $\bullet$ For any $n\in\N_0$, we define the seminorm $ |f|_{n,\infty} := \sup_{x\in\spt(f)} \max_{\alpha\in \N_0^d, \, |\alpha|=n} |\pa_{x}^\alpha f(x) |$
%\begin{equation}
%		|f|_{n,\infty} := \sup_{|x|\leq \varepsilon^{k}} \max_{\alpha\in \N_0^d, \, |\alpha|=n} |\pa_{x}^\alpha f(x) |
%\end{equation}
%where $f:\R^{m_1}\to \R^{m_2}$;
\\[1mm]
\noindent $\bullet$ $\cC^{s}(\T^{m_1},\R^{m_2}):= \big\{ f:\R^{m_1}\to \R^{m_2} \,:\, \|f \|_{n,\infty} <\infty \ \ \forall\, n\in \N \cup \{ 0 \}, \, n\leq s  \big\}$, $s \in \N \cup \{\infty\}$;
%\\[1mm]
%\noindent $\bullet$ $\cC^{\infty}(\T^{m_1},\R^{m_2}):= \big\{ f:\R^{m_1}\to \R^{m_2} \,:\, |f|_{n,\infty} <\infty \ \ \forall\, n\in\N_0 \big\}$;
%\\[1mm]
%\noindent $\bullet$ $\cC_{\rm div}^{S}(\T^{d},\R^{d}):= \big\{ f\in \cC^{S}(\T^{d},\R^{d}) \,:\, {\rm div} \, f = 0 \big\} $, $S\in \N \cup \{\infty\}$;
\\[1mm]
\noindent $\bullet$ $\cC^{\infty}(X,\R) := \cC^{\infty}(X)$, with $X=\T^d, \R^d$;
\\[1mm] 
\noindent $\bullet$ $a \lesssim b$ stands for $a \leq C b$, for some constant $C>0$;
\\[1mm] 
\noindent $\bullet$ $a \lesssim_{n} b$ stands for $a \leq C_{n} b$, for some constant $C_n>0$ depending on $n$.

%\section{The construction of the solutions}
\section{Proof of Theorem \ref{teorema}}

The scheme follows essentially the one proposed in \cite{EPsTdL}, with the required adaptations. The key starting point is the existence of smooth, compactly supported stationary solutions of the Euler equations. In $d=3$, this is celebrated result by Gravilov \cite{gravilov} (see also \cite{CLV}), whereas in even dimension it has been proved in \cite{EPsTdL}. We recall the statement of the result of the latter.

\begin{prop}\label{gravi.sol}
	{\bf (Smooth stationary Euler flows with compact support - Proposition 2, \cite{EPsTdL}).}
	If $d=3$ or $d\in\N$ is even, there exists a smooth, compactly supported solution $v(x)\in \cC_{\rm div}^\infty(\R^d,\R^d) $, with pressure $p_{v}(x)\in \cC^\infty(\R^d)$, of
	\begin{equation}\label{euler.stat}
		v\cdot \nabla v + \nabla p_{v} = 0 \,, \quad {\rm div} \,v = 0\,, \quad x \in \R^d\,.
	\end{equation}
\end{prop}
 Without any loss of generality, we assume that $\spt(v), \spt(p_v) \subseteq B_{d,1}(0) \subset \R^d$. Then, given $S\in\N$  and for any $k\in\N$, we define the rescaled functions, for any $\varepsilon\in (0,1)$ small enough,
 \begin{equation}
 	v_k(x) := \varepsilon^{(S + 1)(k-1)} v(\varepsilon^{-k}x) \,,\quad p_{v_{k}}(x) := \varepsilon^{2(S + 1)(k-1)} p_{v}(\varepsilon^{-k}x) \,.
 \end{equation}
 A straightforward computation shows that $v_{k}(x)\in \cC_{\rm div}^\infty(\R^d,\R^d) $ is also a solution of \eqref{euler.stat} with pressure $p_{v_{k}}(x)\in \cC^\infty(\R^d)$ with compact support
\begin{equation}
	\spt (v_{k}), \spt(p_{v_k}) \subseteq B_{d,\varepsilon^{k}}(0) \subset \R^d\,,
\end{equation}
and we have the control on the seminorms, for any $n\in\N_0$,
%\begin{equation}
%	\sup_{|x|\leq \varepsilon^{k}} \max_{\alpha\in \N_0^d, \, |\alpha|=n} |\pa_{x}^\alpha v_{k}(x) | \leq C_n \varepsilon^{k(S-n)-S} \,, \quad \sup_{|x|\leq \varepsilon^{k}} \max_{\alpha\in \N_0^d, \, |\alpha|=n} |\pa_{x}^\alpha p_{v_{k}}(x) | \leq C_n \varepsilon^{k(2S-n)-2S} 
%\end{equation}
\begin{equation}\label{rescaled.esti}
	\|v_{k}\|_{n,\infty}  \leq C_n \varepsilon^{k(S + 1-n)-S - 1} \,, \quad \|p_{v_{k}}\|_{n,\infty} \leq C_n \varepsilon^{k(2S + 2-n)-2S - 2} \,,
\end{equation}
for some constant $C_n>0$ independent of $\varepsilon \in (0,1)$ and $k\in\N$.
We remark that, as soon as $n> S + 1$, the seminorms $\|v_{k}\|_{n,\infty}$ start to diverge with respect to $k\to\infty$ as $\varepsilon^{-(n-S - 1)k}$ for $\varepsilon\in (0,1)$, whereas the seminorms $\|p_{v_{k}}\|_{n,\infty}$ start to diverge when $n>2S + 2$. 

Moreover, for $\varepsilon\in (0,1)$ small enough and independent of $k\in\N$, we define the periodicized versions
\begin{equation}
	\overline{v_{k}} (x) := \sum_{q\in \Z^d} v_{k}(x+ 2\pi q) \,, \quad 	\overline{p_{v_{k}}} (x) := \sum_{q\in \Z^d} p_{v_{k}}(x+ 2\pi q)\,, \quad x \in \T^d\,.
\end{equation}

%Let now $m\in\{1,...,d-1\}$ and fix a sequence $(J_{k})_{k\in\N} \in \ell^\infty(\N,\N)$ such that $N_k \leq (d-m)J_k$ for any $k\in\N$.
%Let now $m\in\{1,...,d-1\}$ be fixed as in \eqref{sequence.Nk}.
We recall the sequence $(N_k)_{k\in\N}\in\ell^\infty(\N,\N)$ in \eqref{sequence.Nk} is determined by a fixed $m\in \{1,...,d-1\}$ and a fixed sequence $(J_k)_{k\in\N}\in \ell^\infty(\N,\N)$.
The choice of $m \in \{1,...,d-1\}$ induces the splitting $\T^d = \T^m \times \T^{d-m}$ and we write
%\begin{equation}
%	\T^d \ni x = (x',x'')  \in\T^m \times \T^{d-m}\,.
%\end{equation}
\begin{equation}
	\T^d \ni x = (x',x'')  \in\T^m \times \T^{d-m}\,,\quad  \nabla=(\nabla',\nabla'') := (\nabla_{x'},\nabla_{x''}) \,.
\end{equation}
We select a sequence of points $(\ty_{k,j})_{k\in\N, \, j=1,..,J_k}\subset \T^m$ with the properties that:
\\[1mm]
\noindent {\bf (A)} For any $k_1,k_2\in\N$, $j_1=1,...,J_{k_1}$, $j_2=1,...,J_{k_2}$, with $(k_1,j_1)\neq (k_2,j_2)$, we have
\begin{equation}
	\overline{B_{m,2\varepsilon^{k_1}}(\ty_{k_1,j_1})} \cap \overline{B_{m,2\varepsilon^{k_2}}(\ty_{k_2,j_2})} = \emptyset \,;
\end{equation}
%\noindent {\bf (B)} We have $\big|\T^m \setminus \big( \bigcup_{k\in\N} \bigcup_{j=1}^{J_k} \overline{B_{2\varepsilon^{k}}(\ty_{k,j})} \big)\big|\geq \frac12 \tA_{m} >0$, where $\tA_{m} = |\T^m|=(2\pi)^m$.
\noindent {\bf (B)} We have $\big|\T^m \setminus \big( \bigcup_{k\in\N} \bigcup_{j=1}^{J_k} \overline{B_{m,2\varepsilon^{k}}(\ty_{k,j})} \big)\big|\geq \frac23 |\T^m|>0$.
\\[1mm]
\indent The existence of such sequence of points with these desired properties is proved in the following lemma.
%\begin{lem}
%	There exists $\varepsilon_{0}\in (0,1)$ small enough and a choice of distinct points $(\ty_{k,j})_{k\in\N, \, j=1,..,J_k}\subset \T^m$, such that, for $\varepsilon\in (0,\varepsilon_{0})$, the conditions {\bf (A)} and {\bf (B)} are satisfies.
%\end{lem}
\begin{lem}
	There exist $\varepsilon_{0}=\varepsilon_{0}\big(m,\| (J_k) \|_{\ell^\infty}\big)\in (0,1)$ small enough and a choice of infinitely many distinct points $(\ty_{k,j})_{k\in\N, \, j=1,..,J_k}\subset \T^m$, such that the following holds. For $\varepsilon>0$, we define iteratively the sets
	\begin{equation}\label{E.sets}
		E_0:= \emptyset \,, \quad E_k :=  E_{k-1} \cup \bigcup_{j=1}^{J_k}\overline{ B_{m,2\varepsilon^{k}}(\ty_{k,j}) } \,, \quad k \in \N \,.
	\end{equation}
	Then, for any $\varepsilon\in (0,\varepsilon_{0})$ and for any $k\in\N$, we have:
	%	 \begin{equation}
		%	 	E_{k} \subsetneq \T^m \setminus \bigcup_{n=1}^{k} E_{n-1} \,, \quad  	\overline{B_{2\varepsilon^{k}}(\ty_{k_1,j_1})} \cap \overline{B_{2\varepsilon^{k}}(\ty_{k_2,j_2})} = \emptyset \,,\quad  \forall j_1,j_2 =1,...,J_k \,,
		%	 \end{equation}
	%	 and
	%	 \begin{equation}
		%	 	|\T^m \setminus E_{k} | =\Big( 1- \sum_{n=1}^{k} 2^{-n-1}  \Big) |\T^m| \,.
		%	 \end{equation}
	\\[1mm]
	%	 \noindent $(i)$ $ E_{k} \subsetneq \T^m \setminus \bigcup_{n=1}^{k} E_{n-1} $;
	\noindent $(i)$ $E_{k-1} \cap \bigcup_{j=1}^{J_k} \overline{ B_{m,2\varepsilon^{k}}(\ty_{k,j})} =\emptyset$
	\\[1mm]
	\noindent $(ii)$ $ \overline{B_{m,2\varepsilon^{k}}(\ty_{k_1,j_1})} \cap \overline{B_{m,2\varepsilon^{k}}(\ty_{k_2,j_2})} = \emptyset $ for any $j_1,j_2 =1,...,J_k$;
	\\[1mm]
	%	 \noindent $(iii)$ $|\T^m \setminus \bigcup_{n=1}^{k} E_{n} | =\Big( 1- \sum_{n=1}^{k} 2^{-n-1}  \Big) |\T^m| $.
	\noindent $(iii)$ $\T^m \setminus E_k $ is open and $|\T^m \setminus  E_{k} | \geq \Big( 1- \sum_{n=1}^{k} 4^{-n}  \Big) |\T^m| $.
	\\[1mm]
	As a consequence, conditions {\bf (A)} and {\bf (B)} are satisfied.
\end{lem}
%\begin{proof}
%	Pick any choice of distinct points $\ty_{1,1},...,\ty_{1,J_1} \in \T^m$ and we define the set $E_1 := \T^m \setminus \bigcup_{j=1}^{J_1} \overline{B_{2\varepsilon}(\ty_{1,j})}$. It is clear there exists $\varepsilon_{0}= \varepsilon_{0} (m,J_1)\in (0,1)$ small enough such that, for any $\varepsilon \in (0,\varepsilon_{0})$, the closed balls $\big(\overline{B_{2\varepsilon}(\ty_{1,j})}\big)_{j=1,...,J_1}$ are pairwise disjoint  and , 
%	\begin{equation}
	%		|E_1| = (2\pi)^m(1 - C_m J_1 \varepsilon^m ) \geq \tfrac34 (2\pi)^m >0\,,
	%	\end{equation}
%	as soon as $\varepsilon<\varepsilon_{0} \leq  \big( 4C_m J_1 \big)^{-1/m}$.
%	We define inductively the sequence of sets  and $E_{k+1} := E_{k} \setminus\bigcup_{j=1}^{J_{k+1}} \overline{B_{2\varepsilon^{k+1}}(\ty_{k+1,j})}$, for any $k\geq 2$.
%%	\begin{equation}
	%%		E_1 := \T^m \setminus \bigcup_{j=1}^{J_1} \overline{B_{2\varepsilon}(\ty_{1,j})}
	%%	\end{equation}
%\end{proof}
\begin{proof}
	In the following, we use that $|B_{m,2r}(\ty) | = C_{m} r^{m} |\T^m| $, where the explicit constant $C_m := \big( \pi^{m/2} \Gamma(\tfrac{m}{2}+1) \big)^{-1}\in (0,1)$ depends only on the dimension $m\in\N$, where $\Gamma$ is the Euler Gamma function.
	
	We argue by induction. Let $k=1$.
	We pick an arbitrary choice of distinct points $\ty_{1,1},...,\ty_{1,J_1} \in \T^m$ and, for $\varepsilon>0$, we define the set $E_1 := \bigcup_{j=1}^{J_1} \overline{B_{m,2\varepsilon}(\ty_{1,j})}$. By \eqref{E.sets}, item $(i)$ is automatically satisfied for $k=1$. We define $\varepsilon_{1,1}:= \tfrac14 \min \{|\ty_{1,j_1}-\ty_{1,j_2}| \,:\, 1\leq j_1 < j_2\leq J_1  \}$. Then, for any $\varepsilon \in (0,\varepsilon_{1,1})$, the closed balls $\big(\overline{B_{m,2\varepsilon}(\ty_{1,j})}\big)_{j=1,...,J_1}$ are pairwise disjoint, that is, item $(ii)$ is satisfied when $k=1$. Clearly, we also have that $\T^m \setminus E_1$ is open, since $E_1$ is a finite union of closed sets. Moreover, using that $C_m\in(0,1)$, we compute
	\begin{equation}
		|\T^m \setminus E_1| =(1 - C_m J_1 \varepsilon^m ) |\T^m| \geq (1 - J_1 \varepsilon^m ) |\T^m|\geq  \tfrac34 |\T^m| \,,
	\end{equation}
	as soon as $\varepsilon<\varepsilon_{1,2} :=  \big( 4 J_1 \big)^{-1/m}$. Therefore, choosing $\varepsilon_{0} \leq \min\{\varepsilon_{1,1},\varepsilon_{1,2} \}$, we conclude that $(i)$, $(ii)$ and $(iii)$ are satisfied for $k=1$.
	
	We now assume that the claims  $(i)$, $(ii)$ and $(iii)$ are satisfied for some $k\in\N$ and we prove them for $k+1$. We set
	\begin{equation}\label{scelta.eps0}
		\varepsilon_{0} := \min \{ \varepsilon_{1,1} , (4  \| (J_k) \|_{\ell^\infty})^{-1/m} \} \leq \min \{ \varepsilon_{1,1},\varepsilon_{1,2}\} \,.
	\end{equation}
	By \eqref{scelta.eps0}, there exist $J_{k+1}$ distinct points $\ty_{k+1,1},...,\ty_{k+1,J_{k+1}} \in \T^m \setminus E_{k}$,  with  $J_{k+1} \leq \| (J_k) \|_{\ell^\infty}$,  such that, for any $\varepsilon \in (0,\varepsilon_{0})$ we have that the $J_{k+1}$ balls $\overline{B_{m,2\varepsilon^{k+1}}(\ty_{k+1,1})}$,  ..., $\overline{B_{m,2\varepsilon^{k+1}}(\ty_{k+1,J_{k+1}})}$ are contained in $\T^m \setminus E_k$ and they are disjoint, namely they satisfy items $(i)$ and $(ii)$ at the step $k+1$. This follows from the fact that, by the induction assumption on $(iii)$, we have that $\T^m \setminus E_k$ is open with measure $| \T^m \setminus E_k| > \tfrac23 |\T^m|$, whereas the measure of the finite union of closed disjoint balls is estimated, for any $\varepsilon \in (0,\varepsilon_{0})$ with $\varepsilon_{0}$ as in \eqref{scelta.eps0}, by
	\begin{equation}\label{measure.balls}
		\begin{aligned}
			\Big|   \bigcup_{j=1}^{J_{k+1}} \overline{B_{m,2\varepsilon^{k+1}}(\ty_{k+1,j})}\, \Big| & = C_m J_{k+1} \varepsilon^{(k+1)m} |\T^m| \\
			&\leq  \frac{1}{4^{k+1}} \frac{J_{k+1}}{\| (J_k) \|_{\ell^\infty}^{k+1}} |\T^m| \leq  \frac{1}{4^{k+1}} |\T^m| < \frac23 |\T^m| \,,
		\end{aligned}
	\end{equation}
	which implies the existence of the $J_{k+1}$ points $\ty_{k+1,1},...,\ty_{k+1,J_{k+1}}$ in the open and bounded set $\T^m\setminus E_{k}$ with the desired properties. Therefore, let $E_{k+1}$ be defined as in \eqref{E.sets}. Clearly, $E_{k+1}$ is closed, which also implies that $\T^m\setminus E_{k+1}$ is open. By \eqref{measure.balls} and item $(ii)$ at the step $k+1$, we also deduce that,
	\begin{equation}
		|\T^m \setminus E_{k+1} | = |\T^m \setminus E_{k}| - 	\Big|   \bigcup_{j=1}^{J_{k+1}} \overline{B_{m,2\varepsilon^{k+1}}(\ty_{k+1,j})}\, \Big| \geq \Big( 1 - \sum_{n=1}^k 4^{-n} - 4^{-(k+1)} \Big)|\T^m|  \,,
	\end{equation}
	which is indeed the estimate in item $(iii)$ at the step $k+1$. This closes the induction argument and concludes the proof.
	%	\begin{equation}
		%		\Big| (\T^m \setminus E_{k})\setminus \bigcup_{j=1}^{J_{k+1}} \overline{B_{2\varepsilon^{k+1}}(\ty_{k+1,j})}  \Big| = \Big( 1- \sum_{n=1}^{k} 2^{-n-1}   \Big)|\T^m| -  C_m 
		%	\end{equation}
\end{proof}

%As a last preliminary, for any $k\in\N$ we denote $\wt\nu_{k} := \imath_{k}(\nu_{k}) \in\R^{(d-m)J_k}$, where $\imath_{k}:\R^{N_k}\to \R^{(d-m)J_k}$ is the canonical embedding such that, when $N_k=(d-m)J_k$, we have $\imath_{k} = {\rm id}_{\R^{N_k}}$.
As a last preliminary, we take a sequence of frequency vectors $\nu = (\nu_{k})_{k\in\N}\in (\R^{N_k})_{k\in\N}$, where  $\nu_{k} = (\nu_{k,1},...,\nu_{k,J_k}) \in \R^{N_k}$, with $\nu_{k,j}\in\R^{d-m}$, recalling \eqref{sequence.Nk}.
We now define the vector field
\begin{equation}\label{solution.u}
	u(t,x) := \sum_{k=1}^\infty u_{k}(t,x) \,, \quad u_k(t,x):= \sum_{j=1}^{J_k} \overline{v_{k,j}}(t,x) + w_k(x)
\end{equation}
with pressure
\begin{equation}\label{pressure.u}
	p_{u}(t,x) :=  \sum_{k=1}^\infty p_{u_{k}}(t,x) \,, \quad p_{u_{k}}(t,x):= \sum_{j=1}^{J_k} \overline{p_{k,j}}(t,x) \,,
\end{equation}
where
\begin{equation}\label{vkj}
	\begin{aligned}
		\overline{v_{k,j}}(t,x) &:= \overline{v_{k}} (x'-\ty_{k,j}, x'' - \nu_{k,j}t)\,, \\ \overline{p_{k,j}}(t,x) & := \overline{p_{v_{k}}} (x'-\ty_{k,j}, x'' - \nu_{k,j}t)\,,
	\end{aligned}
\end{equation}
and
\begin{equation}\label{w.vf}
	w_k(x) = (0,F_k(x')) \,, \quad F_{k} : \T^m \to \R^{d-m}\,.
\end{equation}
Note that, no matter the choice of $F_k(x')$ sufficiently smooth is, the vector field $w_k(x)$ is a stationary solution with constant pressure of the Euler equations \eqref{euler.eq}, namely we have
\begin{equation}\label{wk.eq}
	w_k \cdot \nabla w_k = 0 \,, \quad {\rm div}\,w_k = 0 \,.
\end{equation}
To make sure that $u(t,x)$ is indeed a solution of \eqref{euler.eq}, we need to specify the functions $F_k(x')$. In particular, we choose
\begin{equation}\label{Fk}
	F_k(x') := \sum_{j=1}^{J_k} \nu_{k,j} \,\chi_{k} (|x'-\ty_{k,j}|) \,,
\end{equation}
where $\chi_{k}(r) \in \cC^{\infty}(\R)$ is an even cut-off function satisfying
\begin{equation}\label{cutoff.k}
\begin{aligned}
	& \chi_{k}(r) = 1 \quad \text{when} \ \ |r|< \varepsilon^k\,, \quad 	\chi_{k}(r) = 0 \quad \text{when} \ \ |r|> 2 \varepsilon^k\,, \quad \chi_{k} \in [0,1] \,, \\
	& |\partial_r^n \chi_k(r) | \leq C_n \varepsilon^{- k n}, \quad \forall \, r \in \R, \quad n \in \N \cup \{ 0 \}\,,
	\end{aligned}
\end{equation}
for some constant $C_n > 0$. For each $k\in\N$, we have that $F_k \in \cC^\infty(\T^m,\R^{d-m})$ and that the vector field $w_k(x)$ is locally equal to $(0,\nu_{k,j})$ when $x\in \spt(\overline{v_{k,j}})$. Note that each pair $(u_{k}(t,x),p_{u_k}(t,x))$ defined above by \eqref{solution.u}-\eqref{Fk} has actually the form of a quasi-periodic solution of the Euler equations \eqref{euler.eq}  as provided in \cite{EPsTdL}, which has been reproduced here on supports of scale $\varepsilon^{k}$. Moreover, by construction and by {\bf (A)}, the support in space of $(u_{k}(t,x),p_{u_k}(t,x))$ is in $\big(\bigcup_{j=1}^{J_k} B_{m,2\varepsilon^k}(\ty_{k,j}) \big)\times \T^{d-m}$ and it is disjoint from the one of $(u_{k'}(t,x),p_{u_{k'}}(t,x))$ for any $k'\neq k$. We use these properties to check that the pair $(u(t,x),p_{u}(t,x))$  in \eqref{solution.u}-\eqref{pressure.u} is indeed a solution of \eqref{euler.eq} as well.

First, we prove that each pair $(u_k(t,x), p_{u_k}(t,x))$ is a solution of \eqref{euler.eq} and we provide estimates on the seminorms.

\begin{lem}\label{lemma.uk}
	Assume that $\nu = (\nu_{k})_{k\in\N}$ satisfies \eqref{small.freq}.
	For each $k\in\N$, the vector field $u_k(t,x) $ is in $ \cC_{\rm div}^{\infty}(\T^d,\R^{d})$, with pressure $p_{u_{k}}(t,x)$ in $\cC^\infty(\T^d)$, is a solution of the Euler equations \eqref{euler.eq}, namely
	\begin{equation}\label{euler.uk}
		\pa_{t} u_k + u_k \cdot \nabla u_k + \nabla p_{u_k} = 0 \,, \quad {\rm div }\, u_k = 0\,,
	\end{equation}
	compactly supported in space in $\bigcup_{j=1}^{J_k} B_{m,2\varepsilon^k}(\ty_{k,j}) \times \T^{d-m}$. Moreover, we have the estimates, for any integer $n \geq 0$,
%	\begin{equation}\label{uk.esti}
%			\sup_{t\in\R} \|u_k(t,\cdot\,)\|_{n,\infty} \leq C_n \varepsilon^{k(S-n)-S} \,, \quad  	\sup_{t\in\R} \|p_{u_k}(t,\cdot\,)\|_{n,\infty} \leq C_n \varepsilon^{k(2S-n)-2S} \,,
%	\end{equation}
\begin{equation}\label{uk.esti}
	\begin{aligned}
		\sup_{t\in\R} \|u_k(t,\cdot\,)\|_{n,\infty} & \leq C_n \varepsilon^{k(S + 1-n)-S - 1} \,, \\
		\sup_{t\in\R} \|\pa_{t}u_k(t,\cdot\,)\|_{n,\infty} & \leq C_n \varepsilon^{k(2S + 2-(n+1))-2S - 2} \,, \\
		 \sup_{t\in\R} \|p_{u_k}(t,\cdot\,)\|_{n,\infty}& \leq C_n \varepsilon^{k(2S+ 2-n)-2S - 2} \,,
	\end{aligned}
\end{equation}
	for some constant $C_n>0$ independent of $\varepsilon\in (0,1)$ and of $k\in\N$.
\end{lem}
\begin{proof}
	By \eqref{solution.u}-\eqref{Fk}, Proposition \ref{gravi.sol} and by {\bf (A)}, we compute
	\begin{equation}\label{ditu}
		\begin{aligned}
				\pa_{t} u_k & = - \sum_{j=1}^{J_k} \nu_{k,j} \cdot \nabla'' \overline{v_{k}}(x'-\ty_{k,j},x''-\nu_{k,j} t) \,,
		\end{aligned}
	\end{equation}
	and, using \eqref{ditu},
	\begin{equation}
		\begin{aligned}
			u_k \cdot \nabla u_k & = \Big( \sum_{j=1}^{J_k} \overline{v_{k,j}} + w_k \Big) \cdot  \Big( \sum_{j=1}^{J_k}\nabla \overline{v_{k,j}} + \nabla w_k \Big) \\
			& = \sum_{j=1}^{J_k} \overline{v_{k,j}}\cdot \nabla \overline{v_{k,j}} + \sum_{j=1}^{J_{k}} \big(  \overline{v_{k,j}} \cdot \nabla w_k + w_k \cdot \nabla \overline{v_{k,j}} \big) + w_k \cdot \nabla w_k \\
			& = -  \sum_{j=1}^{J_k}  \nabla \overline{p_{k,j}} + \sum_{j=1}^{J_k} \big(0+ \nu_{k,j}\cdot \nabla'' \overline{v_{k}} (x'-\ty_{k,j},x''-\wt\nu_{k,j}t) \big)  +0 \\
			& = - \nabla p_{u_k} - \pa_{t} u_k \,.
		\end{aligned}
	\end{equation}
	This, together with the fact that ${\rm div}\,\overline{v_{k,j}} =0$ for any $j=1,...,J_k$ by \eqref{vkj} and Proposition \ref{gravi.sol}, concludes the proof of \eqref{euler.uk}. \\
	To prove the estimates \eqref{uk.esti}, we first note that, by \eqref{cutoff.k}, we have $\|\chi_{k}\|_{n,\infty} \leq C_n \varepsilon^{-nk}$ for any integer $n \geq 0$. Moreover, by \eqref{small.freq}, we have that $|\nu_{k}|\leq C \varepsilon^{(S + 1)(k-1)}$, for some constant $C>0$ independent of $k\in\N$. Therefore, we deduce, that $w_k(x)$ in \eqref{w.vf}-\eqref{Fk} satisfies $\|w_{k}\|_{n,\infty} \leq C_n \varepsilon^{(S + 1-k)n-S - 1}  $  for any integer $n\geq 0$, recalling that $(J_k)_{k\in\N}\in \ell^\infty(\N,\N)$. Furthermore, by \eqref{ditu} and using the fact that each $\overline{v_{k,j}}(t,x)$ is supported in space on the cylinder $B_{m,2\varepsilon^k}(\ty_{k,j})\times \T^{d-m}$, disjoint for any $j'\neq j$  from the cylinder $ B_{m,2\varepsilon^k}(\ty_{k,j'}) \times \T^{d-m}$ supporting $\overline{v_{k,j'}}(t,x)$, we estimate, for any integer $n \geq 0$ and uniformly in $t\in\R$,
	\begin{equation}\label{stima.ditu}
	\begin{aligned}
			\| \pa_{t}u_{k}(t,\,\cdot) \|_{n+1, \infty} & \sup_{j=1,...,J_k}|\nu_{k,j}| \| \nabla'' \overline{v_{k}}(x'-\ty_{k,j},x''-\wt\nu_{k,j} t)   \|_{n,\infty} \\
			 &\lesssim  \varepsilon^{(S+1)(k-1)} \| v_{k} \|_{n+1,\infty}  \\
			 & \lesssim_{n}   \varepsilon^{(S + 1)(k-1)} \varepsilon^{k(S + 1-(n+1))-S - 1} \,.
	\end{aligned}
	\end{equation}
	 Collecting together \eqref{solution.u}, \eqref{pressure.u}, \eqref{w.vf} and estimates \eqref{rescaled.esti}, \eqref{stima.ditu}, we obtain the estimates \eqref{uk.esti} and the proof is concluded.
\end{proof}

We now show that $u(t,x)$ in \eqref{solution.u} solves the Euler system \eqref{euler.eq} and that it has the desired regularity and estimates. 

\begin{prop}
	The vector field $u(t,x)$ in \eqref{solution.u}, with pressure $p_{u}(t,x)$ as in \eqref{pressure.u}, is a solution of the Euler equations \eqref{euler.eq}. Moreover, assuming that $\nu = (\nu_{k})_{k\in\N}$ satisfies \eqref{small.freq}, we have $u(t,\,\cdot\,) \in \cC_{\rm div}^{S + 1}(\T^d,\R^d)$, $\partial_t u(t, \cdot) \in \cC_{\rm div}^{2 S + 1}(\T^d,\R^d)$ and $p_{u}(t, \cdot) \in \cC^{2 S + 2}(\T^d)$, with estimates, 
	\begin{equation}\label{u.final.esti}
		\begin{aligned}
		\sup_{t\in\R} \|u(t,\cdot\,)\|_{n,\infty}& \leq C_n \varepsilon^{-S - 1} \,, \quad \forall\, n=0,1,...,S + 1\,, \\
		\sup_{t\in\R} \|\pa_{t} u(t,\cdot\,)\|_{n,\infty}& \leq C_n \varepsilon^{-2S - 2} \,, \quad \forall\, n=0,1,...,2S + 1\,, \\
		\sup_{t\in\R} \|p_{u}(t,\cdot\,)\|_{n,\infty}& \leq C_n \varepsilon^{-2S - 2}\,, \quad \forall\, n=0,1,...,2S + 2 \,.
		\end{aligned}
	\end{equation}
\end{prop}
\begin{proof}
	By Lemma \ref{lemma.uk} and by {\bf (A)}, each vector field of the sequence $(u_k(t,x))_{k\in\N}$ is compactly supported in space and all these supports are pairwise disjoint. We use this properties and the fact that each $u_k(t,x)$ in solves \eqref{euler.uk} to compute, with $u(t,x)$ and $p_u(t,x) $ as in  \eqref{solution.u}, \eqref{pressure.u},
	\begin{equation}
		\begin{aligned}
			 u\cdot \nabla u &= \sum_{k=1}^\infty u_k \cdot \sum_{k=1}^\infty \nabla u_k = \sum_{k=1}^\infty u_k \cdot \nabla u_k = \sum_{k=1}^\infty \big( - \nabla p_{u_k}  - \pa_{t}u_k\big) = - \nabla p_{u} -\pa_{t} u \,, \\
			{\rm div} \, u & =  \sum_{k=1}^\infty {\rm div} \, u_k = 0 \,,
		\end{aligned}
	\end{equation}
	which indeed proves \eqref{euler.eq}. It remains to prove the finite regularity of the solution. By \eqref{solution.u}, \eqref{uk.esti}, since the support in space are pairwise disjoint, we have that, for any $n=0,1,...,S + 1$,
	\begin{equation}
		\sup_{t\in\R} \|u(t,\cdot\,)\|_{n,\infty} = \sup_{t\in\R} \sup_{k\in\N} \|u_k(t,\,\cdot\,)\|_{n,\infty} \leq C_n \varepsilon^{-S - 1} \sup_{k\in\N} \varepsilon^{k(S + 1-n)} \leq C_n \varepsilon^{-S - 1}\,.
	\end{equation}
	The estimates for $ \partial_t u(t, \cdot), p_u(t, \cdot)$ can be proved similarly and we omit them. Hence \eqref{u.final.esti} follows. This concludes the proof.
\end{proof}

In order to conclude the proof of Theorem \ref{teorema}, it remains to show the existence of the embedding $U:(\T^{N_k})_{k\in\N} \to \cC_{\rm div}^S(\T^d,\R^d)$. 
%For any $k\in\N$, recalling that $N_k\leq (d-m)J_k$, we consider a linear embedding $\cN_{k}:\T^{N_k} \to \T^{(d-m)J_k}$ of the form
%\begin{equation}
%	\T^{N_k}\ni \vartheta_{k} \to \cN_{k}(\vartheta_{k}):=\big( \cN_{k,1}(\vartheta_{k}),..., \cN_{k,J_k} (\vartheta_{k})\big) \,,\quad \cN_{k,j}: \T^{N_k} \to \T^{d-m}\,,
%\end{equation}
%where the linearity is meant at the level of the lift to the universal cover $\R^{(d-m)J_{k}}$. Let $ \overline{\cN}_{k}:= \di_{\vartheta_{k}} \cN_{k}(0):\R^{N_k} \to \R^{(d-m)J_k}$. We now define the map $\cN : (\T^{N_k})_{k\in\N} \to (\T^{(d-m)J_k})_{k\in\N}$ by 
%\begin{equation}
%	(\T^{N_k})_{k\in\N} \ni (\vartheta_{k})_{k\in \N} = \vartheta \mapsto \cN(\vartheta) = \big(\cN_{k}(\vartheta_{k}) \big)_{k\in\N} \in (\T^{(d-m)J_k})_{k\in\N} \,.
%\end{equation}
%Abusing notation, we denote 
%%\begin{equation}
%%	\overline{\cN} := (\overline{\cN}_k)_{k\in\N} := \di_{\theta} \cN(0) : (\R^{N_k})_{k\in\N} \to (\R^{(d-m)J_k})_{k\in\N}\,,
%%\end{equation}
%\begin{equation}
%	 \di_{\vartheta} \cN(0)=:(\overline{\cN}_k)_{k\in\N}  : (\R^{N_k})_{k\in\N} \to (\R^{(d-m)J_k})_{k\in\N}\,,
%\end{equation}
%where the differential $\di_{\vartheta}$ is meant in the sense of Fr\'echet. 
We  define the claimed family of (almost-periodic) solutions as $\theta\to U(\theta+ \nu t)$, where  the embedding $U: (\T^{N_k})_{k\in\N}\to \cC_{\rm div}^S(\T^d,\R^d)$ is given by
\begin{equation}
	\begin{aligned}
		u(t,x) &:= U(\theta+\nu t)(x) = \sum_{k\in\N} u_k(t,x) = \sum_{k\in \N} U_k(\theta_{k} + \nu_{k} t)(x) \,, 
	\end{aligned}
\end{equation}
with each $ U_k: \T^{N_k}\to \cC_{\rm div}^\infty(\T^d,\R^d) \subset \cC_{\rm div}^S(\T^d,\R^d)$. for $k\in\N$, given by
\begin{equation}
	\begin{aligned}
		& U_k(\theta_{k} + \nu_{k} t)(x)  := \sum_{j=1}^{J_k} \overline{v_k} \big( x' - \ty_{k,j} , x'' - \theta_{k,j}- \nu_{k,j} t \big) + \sum_{j=1}^{J_k} \big( 0, \nu_{k,j}\, \chi_{k}(|x'-\ty_{k,j}|) \big)\,, \\
		& \theta_k = (\theta_{k, j})_{j = 1, \ldots, J_k} \in \T^{N_n}, \quad \theta_{k, j} \in \T^{d - m}, \quad \forall \,k \in \N\,,\ \  j = 1, \ldots, J_k \,,
	\end{aligned}
\end{equation}
(recall that $N_k = (d - m) J_k$, see \eqref{sequence.Nk}) and  with initial data
%\begin{equation}\label{initial.data}
	\begin{align}
		u_{\theta}(x) &:= U(\theta)(x) = \sum_{k\in\N} u_{\theta_{k}}(x) = \sum_{k\in \N} U_k(\theta_{k})(x) \,,  \label{initial.data} \\
		U_k(\theta_{k})(x) & := \sum_{j=1}^{J_k} \overline{v_k} \big( x' - \ty_{k,j} , x'' - \theta_{k,j}  \big) + \sum_{j=1}^{J_k} \big( 0,  \nu_{k,j}\, \chi_{k}(|x'-\ty_{k,j}|) \big)\,.
	\end{align}
%\end{equation}

As last step, we prove the estimate on the continuity and the differentiability of the embedding $U$.
\begin{prop}
	Assume that $\nu = (\nu_{k})_{k\in\N}$ satisfies \eqref{small.freq}.
	Then the embedding $U: (\T^{N_k})_{k\in\N}\to \cC_{\rm div}^S(\T^d,\R^d)$ is $\cC_b^1$ according to Definition \ref{C1-smooth.def}, with estimates
%		\begin{equation}\label{last.esti}
		\begin{align}
			& \sup_{\vartheta \in (\T^{N_k})_{k\in\N}} \|  U(\vartheta) \|_{n,\infty} \leq C_n  \varepsilon^{-S-1} \,, \quad   0\leq n\leq S\,, \label{last.esti}  \\
			&	\sup_{\vartheta \in (\T^{N_k})_{k\in\N}} \| \di_{\vartheta} U(\vartheta) [\wh\vartheta] \|_{n,\infty} \leq C_n \varepsilon^{-S-1} | \wh\vartheta |_{\infty}  \quad \forall\,\wh \vartheta \in (\R^{N_k})_{k\in\N}\,,  \quad 0\leq n\leq S \,.
		\end{align}
%	\end{equation}
\end{prop}
\begin{proof}
	The first estimate in \eqref{last.esti} follows by \eqref{initial.data} and \eqref{u.final.esti}. We now prove the second estimate in \eqref{last.esti}. By \eqref{initial.data}, we compute, for any $\vartheta\in (\T^{N_k})_{k\in\N}$ and $\wh\vartheta \in (\R^{N_k})_{k\in\N}$,
	\begin{equation}\label{natale1}
	\di_{\vartheta} U(\vartheta)[\wh\vartheta] = - \sum_{k\in\N} \sum_{j=1}^{J_k} \wh\vartheta_{k} \cdot \nabla'' \overline{v_k} (x'-\ty_{k,j}, x'' - \vartheta_{k,j}) \,.
	\end{equation}
	Therefore, by \eqref{rescaled.esti}, using the fact that each term in the series in \eqref{natale1}  is supported in space on the cylinder $B_{m,2\varepsilon^k}(\ty_{k,j})\times \T^{d-m}$, that these supports are  disjoint one from the other, and that $|\wh\vartheta_{k}|\leq |\wh\vartheta|_{\infty}$ , we obtain, for any $\vartheta\in(\T^{N_k})_{k\in\N}$, $\wh\vartheta \in (\R^{N_k})_{k\in\N}$ and for any $n=0,1,...,S$,
%	\begin{equation}
%	\begin{aligned}
%			\| 	\di_{\vartheta} U(\vartheta)[\wh\nu] \|_{n,\infty} & \leq \sum_{k\in\N} \sum_{j=1}^{J_k} \| \nabla'' \overline{v_k} \|_{n,\infty} |\wh\nu|_\infty   \leq \| (J_k) \|_{\ell^\infty}  \sum_{k\in \N} \| \overline{v_k} \|_{n+1,\infty} |\wh\nu|_{\infty} \\
%			& \leq C_n \sum_{k\in \N} \varepsilon^{k(S + 1-(n+1))-S  - 1} |\wh\nu|_\infty \leq C_n \varepsilon^{-S} |\wh\nu|_\infty\,.
%	\end{aligned}
%	\end{equation}
	\begin{equation}
	\begin{aligned}
			\| 	\di_{\vartheta} U(\vartheta)[\wh\vartheta] \|_{n,\infty} & \leq \sup_{k\in\N} \sup_{j=1,...,J_k} \| \nabla'' \overline{v_k}(\,\cdot\, - \ty_{k,j}, \,\cdot\, -\vartheta_{k,j}) \|_{n,\infty} |\wh\vartheta|_\infty  \\ 
		 & 	\leq \sup_{k\in\N} \sup_{j=1,...,J_k} \|  \overline{v_k}(\,\cdot\, - \ty_{k,j}, \,\cdot\, -\vartheta_{k,j}) \|_{n+1,\infty} |\wh\vartheta|_\infty  \\ 
		& \leq C_n \sup_{k\in\N} \varepsilon^{k(S + 1-(n+1))-S  - 1} |\wh\vartheta|_\infty \leq C_n \varepsilon^{-S-1} |\wh\vartheta|_\infty\,.
	\end{aligned}
\end{equation}
	This implies the claimed estimate and concludes the proof.
\end{proof}

\begin{footnotesize}
	
\end{footnotesize}

\bigskip

\begin{flushright}
	\textbf{Luca Franzoi}
	
	\smallskip
	
	Dipartimento di Matematica ``Federigo Enriques''
	
	Universit\`a degli Studi di Milano
	
	Via Cesare Saldini 50
	
	20133 Milano, Italy
	
	\smallskip 
	
	\texttt{luca.franzoi@unimi.it}
	
\end{flushright}
\begin{flushright}
	\textbf{Riccardo Montalto}
	
	\smallskip
	
	Dipartimento di Matematica ``Federigo Enriques''
	
	Universit\`a degli Studi di Milano
	
	Via Cesare Saldini 50
	
	20133 Milano, Italy
	
	\smallskip 
	
	\texttt{riccardo.montalto@unimi.it}
	
\end{flushright}

 \end{document}